\documentclass[12pt]{article}
\usepackage{amsmath}
\usepackage{graphicx}
\usepackage{enumerate}
\usepackage{natbib,xcolor,times}
\usepackage[hidelinks]{hyperref}
\usepackage{url} 

\newcommand{\blind}{0}

\usepackage{adjustbox}
\usepackage{graphicx}
\usepackage{amsmath}
\usepackage{amssymb}
\usepackage{multirow}
\usepackage[tight,footnotesize]{subfigure}

\usepackage{algcompatible,tcolorbox}
\usepackage{algorithm}
\usepackage[noend]{algpseudocode}

\makeatletter
\def\BState{\State\hskip-\ALG@thistlm}
\makeatother

\usepackage{cleveref}
\crefname{assumption}{Assumption}{Assumptions}
\crefname{remark}{Remark}{Remarks}
\crefname{proposition}{Proposition}{Propositions}
\crefname{theorem}{Theorem}{Theorems}
\crefname{section}{Section}{Section}
\crefname{lemma}{Lemma}{Lemma}
\crefname{algorithm}{Algorithm}{Algorithms}
\crefname{example}{Example}{Examples}
\crefname{figure}{Figure}{Figure}
\crefname{appendix}{Appendix}{Appendix}
\crefname{table}{Table}{Table}
\crefname{equation}{equation}{equation}

\def\bY{\mathbf{Y}}
\def\bX{\mathbf{X}}
\def\bM{\mathbf{M}}
\def\bs{\mathbf{s}}
\def\trans{^{\scriptscriptstyle \sf T}}
\def\bgamma{\boldsymbol{\gamma}}

\usepackage{amsmath}

\DeclareMathOperator*{\argmin}{arg\,min}
\usepackage{xcolor}

\newtheorem{theorem}{Theorem}

\usepackage{setspace}
\definecolor{OliveGreen}{rgb}{0,0.6,0}

\newtheorem{remark}{Remark}

\begin{document}

\title{Clustering sequence data with mixture Markov chains with covariates using multiple simplex constrained optimization routine (MSiCOR)}
\author{
Priyam Das$^{1}$\footnote{Corresponding author.}
\and Deborshee Sen$^{2}$
\and Debsurya De$^3$
\and Jue Hou$^4$
\and Zahra S. H. Abad$^5$
\and Nicole Kim$^6$
\and Zongqi Xia$^{7,8}$
\and Tianxi Cai$^{6,9}$
}
\date{$^1$Department of Biostatistics, Virginia Commonwealth University, USA \\
$^2$Department of Mathematical Sciences, University of Bath, UK \\
$^3$Department of Applied Mathematics and Statistics, Johns Hopkins University, USA \\
$^4$Department of Biostatistics, University of Minnesota, USA \\
$^5$Institute of Health Policy, Management and Evaluation, University of Toronto, Canada\\
$^6$Department of Biostatistics, Harvard T. H. Chan School of Public Health, USA\\
$^7$Department of Neurology, University of Pittsburgh, USA\\
$^8$Department of Biomedical Informatics, University of Pittsburgh, USA\\
$^9$Department of Biomedical Informatics, Harvard Medical School, USA}
\if0\blind
{
\maketitle 

} \fi

\if1\blind
{ \begin{center}
{\large Title}
\end{center}
\medskip
} \fi

\begin{abstract}
\noindent 
Mixture Markov Model (MMM) is a widely used tool to cluster sequences of events coming from a finite state-space. However the MMM likelihood being multi-modal, the challenge remains in its maximization. Although Expectation-Maximization (EM) algorithm remains one of the most popular ways to estimate the MMM parameters, however convergence of EM algorithm is not always guaranteed. 
Given the computational challenges in maximizing the mixture likelihood on the constrained parameter space, we develop a pattern search-based global optimization technique which can optimize any objective function on a collection of simplexes, which is eventually used to maximize MMM likelihood. This is shown to outperform other related global optimization techniques. In simulation experiments, the proposed method is shown to outperform the expectation-maximization (EM) algorithm in the context of MMM estimation performance. The proposed method is applied to cluster Multiple sclerosis (MS) patients based on their treatment sequences of disease-modifying therapies (DMTs). We also propose a novel method to cluster people with MS based on DMT prescriptions and associated clinical features (covariates) using MMM with covariates. Based on the analysis, we divided MS patients into 3 clusters.
Further cluster-specific summaries of relevant covariates indicate patient differences among the clusters. 


\end{abstract}

\section{Introduction} \label{sec1}
Mixture Markov Model (MMM) is a well-known clustering technique in order to cluster event sequences where each event is coming from a state-space given by a collection of finite events. In literature several statistical applications of MMM and Mixture Hidden Markov Markov Model (MHMM) can be found for clustering event sequence data in various fields \citep{gupta2016mixtures,Chi2007,Li2019,Melnykov2016}. A few applications of  MMM can also be found in the field of biostatistics \citep{Haan2017,Das2023}. 
However, one of the challenging aspects of estimating MMM parameters remains in optimizing the mixture likelihood. Due to the multi-modal nature of the mixture likelihood, it is difficult to maximize using derivative-based methods as these tend to converge at local maxima. Although the expectation-maximization (EM) algorithm can be used for parameter estimation for MMM \citep{Helske2019}, its performance largely relies on the initial point solution \citep{Couvreur1997}. Several articles addressed the convergence issues of using EM algorithm for estimating parameters in mixture model \citep{Archambeau2003} which might pose a challenge in obtaining the maximum likelihood estimate (MLE) in case of MMM as well. \cite{gupta2016mixtures} proposed a faster alternative to the EM algorithm for optimizing the MMM likelihood, but this was later shown to under-perform when compared to the EM algorithm in terms of predictive performance. Therefore exploring alternative approaches to maximize MMM likelihood would be helpful in improving the estimation performance, resulting in improved inference and related prediction performances as well. 

As an alternative to EM algorithm and other convex optimization techniques, evolutionary global optimization techniques \citep{Andrzej2006} can be considered as a natural choice in order to address the complex and multi-modal nature of the objective function i.e., the MMM likelihood. For optimizing multi-modal functions, several evolutionary global optimization techniques have been proposed in the last few decades, e.g., `Genetic Algorithm (GA)' \citep{Fraser1957,Bethke1980,Goldberg1989} and `Simulated Annealing (SA)' \citep{Kirkpatrick1983,Granville1994}, which are widely used nowadays. These methods were first developed for unconstrained global optimization, later they were extended for optimizing objective functions with constraints \citep{Smith1994,Reid1996}. Among other non-deterministic global optimization techniques,  Particle Swarm Optimization (PSO) (\cite{Kennedy1995}, \cite{Eberhart1995}) remains popular for unconstrained global optimization. \cite{Torczon1997} proposed a different evolutionary based global optimization technique called `Generalized Pattern Search' (GPS) which is a more generalized version of the Direct search method earlier proposed by \cite{Hooke1961}. In GPS, using an `exploratory moves algorithm' \citep{Torczon1997}, the objective function values at a set of neighboring points around the current solution are evaluated and the point with the best (i.e., the maximum or minimum, based on requirement) objective function value is selected as the updated solution. Later GPS has been extended to solve global optimization problems on several constrained spaces and shown to outperform other related evolutionary based algorithms \citep{Das2021, Das2022}. In this paper, we propose a Pattern search based algorithm to maximize the MMM likelihood, compare the performance of the proposed optimization algorithm to other existing evolutionary global optimization techniques, and eventually demonstrate the estimation performance of the proposed algorithm while maximizing MMM likelihood to outperform the results obtained using EM algorithm based on simulation experiments.

The parameter space of the MMM likelihood for known (finite) number of clusters is composed of a set of transition matrices, initial state probability vectors and the vector of mixture proportions. Now, the parameters of the collection of transition matrices can be also observed as a collection of simplexes since each row of a transition matrix belongs to a simplex space. Also, note that, the mixture proportion vector and each of the initial state probability vectors are simplexes as well. Therefore, the whole parameter space of the MMM likelihood can be expressed as a collection of simplexes. In order to solve the MMM likelihood, we come up with a Pattern search based global optimization technique to solve any objective function whose parameter space is given by a collection of simplexes, called multiple simplex constrained optimization routine (MSiCOR). We also explore some of its theoretical properties. 

As an application to MSiCOR in the context of MMM, we cluster Multiple sclerosis (MS) patients based on their disease-modifying treatment (DMT) prescription history. MS is an autoimmune disease of the the central nervous system (CNS) that leads to neurodegeneration in the brain and spinal cord. In MS, the immune system damages the protective sheath (myelin) that covers nerve fibers, causing dysfunctional communication between the CNS and the rest of the body. For most people with MS (pwMS), the disease starts with relapses resulting in episodes of new or worsening symptoms due to acute focal inflammation of the CNS. MS is generally diagnosed in early adulthood. Major symptoms of MS include muscle stiffness, paralysis, cognitive impairment, fatigue, depression, visual disturbance, balance and gait difficulty, and problems with bladder, bowel, or sexual function. 

Although there is no cure for MS, there are 20 FDA-approved DMTs that reduce inflammatory disease activity and delay disease progression. DMTs can be divided into several mechanistic categories, including interferon-beta, glatiramer acetate, fumarates such as dimethyl fumerate and natalizumab, sphigophine-1-phosphate modulators such as fingolimod, and B-cell depletion agents such as rituximab. Over the course of this chronic disease, pwMS typically receive serial DMTs as monotherapies. Common reasons for switching DMTs include therapeutic failure (for example, relapse), intolerance or adverse events. 

Prior research efforts identified various demographic, clinical, genetic and neuroimaging features associated with disease activity of pwMS \citep{Myhr2001,Barcellos2002}. Further, several research works focused on the time series analysis of magnetic resonance imaging (MRI) intensity \citep{Meier2003} and volumetric medical image sequences \citep{Thirion1999,ghribi2018multiple} of the MS patients. \cite{Garcia-Dominguez2016} showed the dependence of DMT preference on socio-demographic and clinical features. However, to the best of our knowledge, few studies have examined DMT prescription sequence data. These prior studies have focused on either a single best patient-specific DMT option \citep{Grand'Maison2018}, or the criteria for switching or stopping DMTs \citep{Gross2019}. Moreover, emphasis has remained mostly on possible treatment options of a patient at a given time point based on response to previous DMTs. To the best of our knowledge, DMT prescription sequence remains an unexplored clinical space in MS. 

As an example, it is possible to estimate transition probabilities of moving from one DMT to another DMT based on patient history both at a personal and at a population level under certain statistical model assumptions. We hypothesize that the variation among these transition probabilities across patients is governed by clinical, demographic, and/or other factors. Since the choice of DMT prescription largely depends on patient disease status, the prescribed DMT at any given time can be mapped to a set of clinical conditions and observations. 
Thus, analysis of DMT transitions and patient clusters based on how transitions across DMT options occur over individual timelines can provide insights on the associations of the clinical features and observations with DMT prescription sequence.

Different statistical models can be used to cluster event sequences in logitudinal data \citep{Scott2020,Murphy2021}.
One potential approach is marked point processes (MPP; \citealp{Jacobsen2006}). 
However, there is generally a lag between the actual onset of new neurological symptoms (that is, relapses indicative of disease activity) and the observed date of DMT prescription. 
The lag can occur when patients visit their physicians a few days after the onset of symptoms/relapses. Given the typical time that it takes for insurance authorization process, there can be further lags between prescription date and the actual DMT start date.
Thus in the scenario of MS DMTs, incorporating the observed time periods between any two prescription dates into the model might be misleading, making MPPs unsuitable for modeling. However, the order in which a patient has been prescribed DMTs over their observation period is still useful and can be modeled in other ways using state-space models (SSMs). In the literature, hidden Markov chain models (HMMs) have been used to cluster event sequences \citep{Helske2019}. However in our scenario we do not have any hidden underlying variable like in HMMs.

In this paper, we model MS DMT prescription sequences using a mixture of discrete state-space Markov chain models. There are a few notable reasons for clustering pwMS based on DMT sequences. First, over the disease course, pwMS are serially treated with multiple DMTs that might change over time. As such, estimating the transitional probabilities of changing across different pairs of DMTs across MS patients or any sub-populations will have clinical relevance. Second, after the individuals are clustered based on DMT sequences, cluster-specific summary statistics of clinical and demographic features can inform prescription guidance for future patients. Further, patient-specific covariates can also be used for clustering along with DMT sequences. We thus incorporate patient-specific covariates within MMMs to explain cluster memberships as in latent class analysis. Although prior publications reported MMM clustering of sequence data (for example, \citealp{gupta2016mixtures}), to the best of our knowledge, none has considered MMM analysis including subject-level covariates. \cite{Helske2019} describes the outline of the MMM and MHMM clustering, including covariates, but no further extensive simulation studies and case studies were explored using MMM. \cite{Bolano2020} considered a mixture transition distribution-like model to account for covariates in Markovian models with illustration using a 3-state HMM and a covariate with three levels. Of clinical relevance, to the best of our knowledge, no study has clustered MS DMT sequence data using MMMs. In related literature, \cite{Altman2005} applied a HMM on MRI lesion count data for MS patients, though this is substantially different from the proposed Markov chain model where the state-space consists of all possible types of DMTs for MS patients and the underlying model is a MMM.


The rest of this paper is organized as follows. The mixture Markov chain model is described in \cref{sec_likelihood}. \cref{computation} is dedicated to the maximum likelihood estimation process, including its computational aspects. Specifically, MSiCOR is described in \cref{sec.algo}. The comparative performance of MSiCOR along with some existing optimization techniques is evaluated in \cref{sec.benchmark}; and some of its theoretical properties are discussed in \cref{sec_theory}. A simulation study is conducted in \cref{sec_sim} to test the performance of MSiCOR, including comparisons with the EM algorithm in \cref{sec_EM}, in the context of MMM parameter estimation. \cref{sec_real} contains the analysis of MS DMT sequence data along with patient-level clinical data. 
Finally, \cref{sec_discussion} concludes the paper.
\\

\section{Mixture Markov chain model (MMM)} \label{sec_likelihood}

\subsection{MMM likelihood without covariates}
Suppose the observed data consist of medication sequences on $N$ treatments from $K$ patients, $\{\bY_i=(Y_{i,1}, \dots, Y_{i, h_i})\}_{i=1}^K$, where $Y_{i,t} \in \{1, \dots, N\}$ represents the treatment received by patient $i$ at time $t \in \{1, \ldots, h_i\}$, and $h_i$ denotes the length (or the total number of time-points) of the medication sequence for the $i$-th patient. We assume that the medication sequences are generated from $L$ time-homogeneous Markov chain processes, or clusters. Let $\bM = [M_1, \dots, M_L]$ denote $L$ Markov transition matrices corresponding to the $L$ clusters, and $\bs = [s_l(k)]_{N\times L} = [\bs_1, \dots, \bs_L]$ denote the corresponding initial state probability vectors where $\bs_l = [s_l(1), \dots, s_l(N)]\trans$ for $l\in \{1,\ldots, L\}$. Then
\begin{align}
\mathbb{P}(\bY_i \mid z_i = l, \bX_i) & = s_l(Y_{i,1})M_l(Y_{i,1}, Y_{i,2}) \cdots M_l(Y_{i,h_i-1},Y_{i,h_i}),
\label{model1}
\end{align}
where $z_i$ denotes the cluster membership of the $i$-th patient, $M_l(y_t, y_{t+1})$ denotes the transition probability from $y_t$ to $y_{t+1}$ and $s_l(y_t)$ denotes the initial state probability of the state $y_t$ in the $l$-th cluster, for $y_t \in \{1,\ldots, N\}$.  The likelihood is given by 
%
\begin{align} \label{eq.prob2_nocovs}
L_{\bM,\bs} (Y_1, \dots, Y_K) 
& = 
\prod_{i=1}^K \sum_{l=1}^L w_{l} s_l(Y_{i,1}) \, M_l(Y_{i,1}, Y_{i,2}) \cdots M_l(Y_{i,h_i-1},Y_{i,h_i}),
\end{align}
where $\{w_l\}_{l=1}^L$ denotes the mixture probability of the $l$-th cluster
such that $\sum_{l=1}^L{w_{l}} = 1, w_l \geq 0$.

\subsection{MMM likelihood with covariates}
Consider along with medication sequence data, we also have covariate data for those $K$ patients given by $\{\bX_i\}_{i=1}^K$, where $\bX_i$ is a $p$-dimensional covariate with 1 being its first element. We assume that the medication sequences are generated from latent class models of $L$ time-homogeneous Markov chain processes with latent class membership $z$ also depending on the covariates $\bX_i$ but that $\bY_i$ is independent of $\bX_i \mid z$. \eqref{model1} holds true in this case as well, and additionally we assume 
\begin{align}
w_{il} \equiv \mathbb{P}(z_i = l \mid \bX_i) & = \frac{\exp{(\bX_i\trans \bgamma_l)}}{1+\sum_{l'=2}^L\exp{(\bX_i\trans \bgamma_{l'})}}, \quad l = 1, \dots, L,
\label{model2}
\end{align}
where $\Gamma = (\bgamma_1,\dots,\bgamma_L)$ denotes the cluster-specific coefficient vectors of the covariates. First cluster is considered as the baseline cluster, so we take $\bgamma_1 \equiv 0$. Here the likelihood is given by 
%
\begin{align} \label{eq.prob2}
L_{\bM,\bs,\Gamma} (Y_1, \dots, Y_K) 
& = 
\prod_{i=1}^K \sum_{l=1}^L w_{il} s_l(Y_{i,1}) \, M_l(Y_{i,1}, Y_{i,2}) \cdots M_l(Y_{i,h_i-1},Y_{i,h_i}).
\end{align}
The posterior probability of a patient belonging to cluster $l$ can be obtained as
\begin{align}\label{eq.memebership}
\mathbb{P}(z_i = l \mid Y_i, X_i) 
= & 
\frac{\mathbb{P}(Y_i \mid z_i = l, X_i) \, \mathbb{P}(z_i = l \mid X_i)}{\sum_{l'=1}^L \mathbb{P}(Y_i \mid z_i = l', X_i) \, \mathbb{P}(z_i = l' \mid X_i)}.
\end{align}

Once we estimate the cluster-specific parameters, we can use \eqref{eq.memebership} to identify where any given patient belongs. Thus, after estimating the clusters, we may summarize the available statistics for the patients to explore the cluster characteristics and differences. In \cref{sec_real}, after the clusters are estimated based on real data, we summarize the patient relapse rate and rates of a few other relevant medical codes for each cluster. These include the International Classification of Diseases (ICD) code, Current Procedural Terminology (CPT) code, and the Concept Unique Identifiers (CUIs).

\section{Computation}
\label{computation}

\subsection{The estimation problem}

We maximize \cref{eq.prob2} in order to estimate the parameters for each Markov chain component. 
To begin with, the $(d-1)$-dimensional simplex is defined as
\begin{align*}
\Delta^{d-1} 
& = 
\left \{(x_{1}, \dots, x_d) \in \mathbb{R}^{d} \; : \; x_i \geq 0, \; i=1,\dots,d, \; \sum_{i=1}^{d}x_i= 1 \right \}.
\end{align*}
In \cref{eq.prob2}, $s_l$ and each row of $M_l$ belongs to $(N-1)$ dimensional simplexes, and $\boldsymbol{\alpha} \in \Delta^{L-1}$. Apart from the coefficient vector $\Gamma$, the remainder of the parameter space consists of $L(N+1)$ simplexes of size $(N-1)$ and one simplex of size $(L-1)$. 

In order to maximize \cref{eq.prob2}, one may use the EM algorithm \citep{Helske2019}, however, EM algorithm tends to get stuck at poor local solutions due to the possible multi-modal nature of the mixture likelihood. Moreover, the estimation performance of EM also depends on the initial starting point, as discussed in Section \ref{sec1}. An alternative is to apply a direct numerical maximization procedure on the objective function. Widely used algorithms for constrained optimization include interior point (IP) methods \citep{Potra2000}, sequential quadratic programming (SQP; \citealp{Wright2005}), and the Broyden-Fletcher-Goldfarb-Shanno (BFGS) algorithm. However, these algorithms typically require good initial starting points in order to reduce the risk of being trapped at poor local maxima. In addition, popular global optimization techniques include the genetic algorithm (GA; \citealp{Fraser1957}) and simulated annealing (SA; \citealp{Kirkpatrick1983}). 
However, a major disadvantage of these algorithms is that they can be computationally slow, and thus there remains a trade-off between estimation accuracy and computation time.

Among other global optimization methods, a generalized pattern search algorithm was proposed by \cite{Torczon1997}, and variations have been proposed for various constrained parameter spaces \citep{Lewis1999,Lewis2000}. Variants of PS have been shown to outperform existing GA or SA algorithms (see, for example, \citealp{Das2021}). In order to maximize the likelihood, we first develop a Pattern search \citep{Torczon1997} based algorithm to maximize any objective function whose parameter space is given by a collection of simplexes, called  multiple simplex constrained optimization routine (MSiCOR). Then, we further modify it to perform updating steps of the unconstrained coefficient vector $\Gamma$, with an end goal to maximize \cref{eq.prob2} overall. Thus we consider a variation of PS that can be used to optimize objective functions over parameter spaces that are a collection of simplexes and of unconstrained parameter spaces, as is the case for the MMM likelihood with covariates based on MS DMT sequences and patients' clinical data.

%

\subsection{Multiple simplex constrained optimization routine (MSiCOR)} \label{sec.algo}

Consider a non-convex objective function $f : \Delta^{n_1-1} \times \cdots \times \Delta^{n_B-1} \mapsto \mathbb{R}$ which we wish to minimize. 
In pattern search (PS), within an iteration while optimizing over a $M$-dimensional space, $2M$ candidate points in the neighborhood of the current solution are explored. These are obtained by changing one coordinate at a time, both in positive or negative direction, keeping other coordinates unchanged. For example, in case we want to optimize a function $f$ over an unconstrained space $\mathbb{R}^M$, given the step-size $s$ and current solution $\mathbf{z} = (z_1,\dots,z_M)$, the objective function $f$ is evaluated at $2M$ new candidate points, given by $\{\mathbf{z}_i^+\}_{i=1}^M$ and $\{\mathbf{z}_i^-\}_{i=1}^M$ where $\mathbf{z}_i^+ = (z_1,\dots,z_{i-1},z_i+s,z_{i+1},\dots,z_M)$ and $\mathbf{z}_i^- = (z_1,\dots,z_{i-1},z_i-s,z_{i+1},\dots,z_M)$. At each iteration, the best solution out of these $(2M+1)$ points is selected based on the objective function values at those points. In PS, the step-size $s$ is chosen adaptively. Unlike the case of unconstrained pattern search, in our problem, the parameter space is composed of multiple simplexes, and the general PS needs to be modified for our scenario.

MSiCOR consists of several \emph{runs}. 
Iterations are performed within each run until a convergence criteria is met, which is detailed in the sequel. Each run starts from the solution returned by the previous run and attempts to find a better solution, with the initial solution for the first run being user-provided. The algorithm terminates and returns the final solution when the solutions obtained by two consecutive runs are close. Having multiple runs aids in jumping out of local minima.

\subsubsection{Tuning parameters}
Each run depends on the following tuning parameters: initial global step-size $s_{\text{initial}}>0$, step decay rate $\rho>1$, step-size threshold $\phi>0$, and sparsity threshold $\lambda \geq 0$. The values of these tuning parameters are set by the user and are kept unchanged across runs. We consider two additional tuning parameters $\tau_1, \tau_2$ to control the convergence criteria. Finally, the maximum number of iterations within a run and the maximum number of runs can be fixed as $M_{\text{iter}}$ and $M_{\text{run}}$, respectively. 

\subsubsection{Global and local step-sizes} 
\label{sec.global.local.stepsize}

The parameter space consists of multiple unit-simplex blocks. Suppose there are $B$ unit-simplex blocks, and that the $j$th simplex block is $(n_j - 1)$-dimensional and denoted by $\mathbf{P}_j = (p_{j,1}, \dots, p_{j,n_{j}}) \in \Delta^{n_j-1}$ for $j = 1,\dots, B$. The total number of parameters is $M = \sum_{j=1}^B n_j$. Within each run, we consider a global step-size $\eta$ and $2M$ \emph{local} step-sizes $\{\{(s_{j,i}^+, s_{j,i}^-)\}_{i=1}^{n_j}\}_{j=1}^{B}$ which are chosen adaptively depending on the tuning parameter values as well as the improvement in the values of objective function. 
Inside a run, in the first iteration the value of the global step-size is set to be $\eta^{(1)} = s_{\text{initial}}$, where $\eta^{(h)}$ denotes the value of global step-size in the $h$th iteration in a run. The value of $\eta$ remains the same throughout an iteration. At the end of each iteration,
its value either remains same or gets divided by $\rho$ $(>1)$ based on a criteria described later in \cref{sec.loop.termnination}.
At the beginning of an iteration, the values of the local step-sizes $s_{j,i}^+$ and $s_{j,i}^-$ are set to be the value of the global step-size $\eta$ of that iteration. 

\subsubsection{Exploratory movements} 
\label{sec.exploratory.movement}
Suppose the current value of the parameters at the beginning of the $h$th iteration is $\mathbf{P} =\mathbf{P}^{(h)} = (\mathbf{P}_1^{(h)},\dots, \mathbf{P}_B^{(h)})$, where $\mathbf{P}_j^{(h)} = (p_{j,1}^{(h)},\dots,p_{j,n_j}^{(h)}) \in \Delta^{n_j-1}$ for $j=1,\dots, B$. 
During the iteration, the objective function is evaluated at $2M$ feasible points in the neighborhood of $\mathbf{P}^{(h)}$. These feasible points are obtained by taking steps around $\mathbf{P}^{(h)}$ modulated by the local step-sizes $s_{j,i}^{+},s_{j,i}^{-}$ for $j=1,\dots, B$, and $i = 1,\dots, n_j$.
These can be divided into $M$ ``positive'' movements $(j,i,+)$ and $M$ ``negative'' movements $(j,i,-)$ for $j = 1,\dots, B$ and $i = 1,\dots, n_j$. We call a coordinate of a unit-simplex box ``significant'' if its value is greater than $\lambda$. 
Suppose that there are $m_j (< n_j)$ significant positions in the $j$th simplex block $\mathbf{P}_j^{(h)}$ excluding the $i$th position $p_{j,i}^{(h)}$. The $(j,i,+)$th movement consists of updating $p_{j,i}^{(h)}$ to $(p_{j,i}^{(h)} + s_{j,i}^{+})$ and subtracting $s_{j,i}^{+}/m_j$ from the $m_j$ significant positions, thus keeping the sum of the values of the $j$th simplex block one. We then check whether the updated $\mathbf{P}_j^{(h)}$ is in the unit-simplex. If so, we set $\mathbf{P}_j^{(h)}(i,+)$ equal to the updated $\mathbf{P}_j^{(h)}$, and if not (since it is possible if either $p_{j,i}^{(h)} + s_{j,i}^{+}>1$ or at least one of the updated values at the significant positions is negative), we update the local step-size by setting $s_{j,i}^{+} = s_{j,i}^{+}/\rho$ and repeat the same step until the updated $\mathbf{P}_j^{(h)}$ is in unit-simplex. However, we do not allow $s_{j,i}^{+}$ to be smaller than $\phi$. In case $s_{j,i}^{+}$ becomes smaller than $\phi$ (by dividing it multiple times by $\rho)$, we set $\mathbf{P}_j^{(h)}(i,+)=\mathbf{P}_j^{(h)}$. The $(j,i,-)$th movement is performed in a very similar manner by subtracting $s_{j,i}^{-}$ from $p_{j,i}^{(h)}$ and adding $s_{j,i}^{-}/m_j$ to the other significant positions, and we refrain from detailing it here.

\subsubsection{Sparsity control}
\label{sec.sparsity.control}

We incorporate a sparsity control step in order to encourage possibly sparse solutions. For each of the obtained modified simplex blocks $\{\mathbf{P}_j^{(h)}(i,+)\}_{i=1}^{n_j}$ and $\{\mathbf{P}_j^{(h)}(i,-)\}_{i=1}^{n_j}$ for $j = 1,\dots,B$, we set the values of the ``insignificant'' coordinates (that is, coordinates less than $\lambda)$ to zero. This is adjusted for in the ``significant'' positions by incrementing each of them by the same amount in order to keep the sum of all the coordinates to be one.
After the sparsity control step, if the modified $\mathbf{P}_j^{(h)}(i,+)$ (or $\mathbf{P}_j^{(h)}(i,-))$ remains in unit-simplex, we denote it by $\mathbf{\bar{p}}_j^{(h)}(i,+)$ (or $\mathbf{\bar{p}}_j^{(h)}(i,-))$. If not, we set $\mathbf{\bar{p}}_j^{(h)}(i,+) = \mathbf{P}_j^{(h)}(i,+)$ \big(or $\mathbf{\bar{p}}_j^{(h)}(i,-) = \mathbf{P}_j^{(h)}(i,-)$\big).

\begin{remark}
$\lambda$ should be taken relatively large in case of prior knowledge that the final solution is sparse. If not, it can be chosen relatively small or zero. 
\end{remark}

\subsubsection{Selecting the best candidate solution}
\label{sec.best.cadidate}

Corresponding to the modified simplex block $\mathbf{\bar{p}}_j^{(h)}(i,+)$, the candidate solution is given by $\mathbf{P}^{(h)}(j,i,+) = (\mathbf{P}_1^{(h)}, \dots, \mathbf{P}_{j-1}^{(h)}, \mathbf{\bar{p}}_j^{(h)}(i,+), \mathbf{P}_{j+1}^{(h)},\dots, \mathbf{P}_B^{(h)})$. Similarly, corresponding to the modified simplex block $\mathbf{\bar{p}}_j^{(h)}(i,-)$, the candidate solution is given by $\mathbf{P}^{(h)}(j,i,+) = (\mathbf{P}_1^{(h)}, \dots, \mathbf{P}_{j-1}^{(h)}, \mathbf{\bar{p}}_j^{(h)}(i,-), \mathbf{P}_{j+1}^{(h)}, \\ \dots, \mathbf{P}_B^{(h)})$. Note that $\{\mathbf{\bar{p}}_j^{(h)}(i,+)\}_{j=1}^{n_j}$ and $\{\mathbf{\bar{p}}_j^{(h)}(i,-)\}_{j=1}^{n_j}$ belong to the unit-simplex. Thus we obtain $2M$ candidate solution points $\mathbf{P}^{(h)}(j,i,+)$ and $\mathbf{P}^{(h)}(j,i,-)$ for $j = 1,\dots,B$ and $i = 1,\dots,n_j$. The objective function is evaluated at these $2M$ candidate solutions and the best solution point (that is, the solution point where the value of the objective function is the lowest) out of the $(2M+1)$ points including current solution $\mathbf{P}^{(h)}$ is set as the updated solution $\mathbf{P}^{(h+1)}$.

\subsubsection{Loop termination criteria} 
\label{sec.loop.termnination}

As mentioned in \cref{sec.global.local.stepsize}, at the end of each iteration, the value of $\eta$ either remains the same or gets divided by $\rho$. If $|f(\mathbf{P}^{(h+1)}) - f(\mathbf{P}^{(h)})| < \tau_1$ at the end of the $(h+1)$th iteration, we set $\eta = \eta/\rho$, and leave it unchanged otherwise. Moreover, we set the the minimum allowable value of $\eta$ to be $\phi$, and terminate a run if $\eta$ becomes smaller than $\phi$. For example, consider a situation where $\eta > \phi$ at the start of the $h$th iteration within the $R$th run. However, at the end of $h$th iteration, $|f(\mathbf{P}^{(h)}) - f(\mathbf{P}^{(h-1)})| < \tau_1$, and so we set $\eta = \eta/\rho$, but it happens to be that $\eta/\rho < \phi$. We then terminate the $R$th run at the $h$th iteration and return the solution obtained at the end of $h$th iteration as the solution of the $R$th run, which then serves as the starting point for the $(R+1)$th run. Recall that we set $\eta = s_{\text{initial}}$ at the beginning of each run. 
Suppose $\mathbf{\hat{P}}^{(R)}$ denotes the solution returned by the $R$th run. The algorithm terminates when $|f(\mathbf{\hat{P}}^{(R)}) - f(\mathbf{\hat{P}}^{(R-1)})| < \tau_2$ and returns $\mathbf{\hat{P}}^{(R)}$ as the final solution. In \cref{flowchart} we provide a flowchart of the runs executed within MSiCOR. Pseudo-code for MSiCOR is provided in Algorithm \ref{euclid}. 

\begin{algorithm}
\caption{MSiCOR}\label{euclid}
\begin{algorithmic}[1]
\State $R \gets 1$.

\BState \emph{top}:

\State $h \gets 1$ and $\eta^{(0)},\eta^{(1)} \gets s_{\text{initial}}$.

\If{$R = 1$}

\State
$\boldsymbol{\mathbf{P}}^{(0)} \gets \text{Initial guess}$ (Cell/list of simplexes of dimensions $n_1, \dots, n_B$ respectively).

\Else 

\State 
$\boldsymbol{\mathbf{P}}^{(0)} \gets \widehat{\mathbf{P}}^{(R-1)}$.

\EndIf

\While{$(h \leq M_{\text{iter}}$ and $\eta^{(h)} > \phi)$}

\State
$F_1 \gets f(\mathbf{P}^{(h-1)})$ and $\eta \gets \eta^{(h-1)}$.

\For{$j=1,\dots,B$}

\For{$v = 1,\dots,2n_j$}

\State 
$i \gets [(v+1)/2]$.

\State 
$\mathbf{q}_{j,v} \gets \mathbf{P}_j^{(h-1)}$ and $\mathbf{q}_v^{\text{temp}} \gets \mathbf{P}_j^{(h-1)}$.

\State 
$\Lambda \gets \text{which}(p_{j,l}^{(h-1)} < \lambda), l \in \{1,\dots,n_j\} \setminus \{i\}$.

\State 
$\Gamma \gets \text{which}(p_{j,l}^{(h-1)} \geq \lambda), l \in \{1,\dots,n_j\} \setminus \{i\}$.

\State 
$s_{j,v} \gets (-1)^v \eta^{(h)}$.

\If{$(n(\Gamma) > 0)$}

\State 
$garbage \gets \Sigma(\mathbf{q}^{\text{temp}}_v(\Lambda))$.

\State 
$\mathbf{q}_{j,v}(i) \gets \mathbf{q}^{\text{temp}}_v(i) + s_{j,v}$
and
$\mathbf{q}_{j,v}(\Gamma) \gets \mathbf{q}^{\text{temp}}_v(i) - s_{j,v}/n(\Gamma)+garbage/n(\Gamma)$
and 
$\mathbf{q}_{j,v}(\Lambda) \gets 0$.

\While{$(\mathbf{q}_{j,v} \not\in \Delta^{n_j-1}$ and $|s_{j,v}|>\phi)$}

\State 
$s_{j,v} \gets s_{j,v}/\rho$.

\State
$\mathbf{q}_{j,v}(i) \gets \mathbf{q}^{\text{temp}}_v(i) + s_{j,v}$
and
$\mathbf{q}_{j,v}(\Gamma) \gets \mathbf{q}^{\text{temp}}_v(i) - s_{j,v}/n(\Gamma)+garbage/n(\Gamma)$.

\EndWhile
\Else

\State 
$\mathbf{q}_{j,v}(i) \gets 1$ and $\mathbf{q}_{j,v}(\Lambda) \gets 0$.
\EndIf

\State
\textbf{if} $(\mathbf{q}_{j,v} \in \Delta^{n_j-1})$ \textbf{then} $f_{j,v} \gets f(\mathbf{P}_1^{(h-1)},\dots, \mathbf{P}_{i-1}^{(h-1)},\mathbf{q}_{j,v}, \mathbf{P}_{i+1}^{(h-1)},\dots, \mathbf{P}_{n_j}^{(h-1)})$
\textbf{else} $f_{j,v} \gets F_1$.

\EndFor
\EndFor

\State
$(j_{\text{best}}, v_{\text{best}}) \gets \argmin_{j,v} f_{j,v}$, over $j=1,\dots,B, v = 1,\dots, n_j$.

\State
$F_2 \gets f_{j_{\text{best}},v_{\text{best}}}$.

\State 
$\mathbf{P}^{(h)} \gets \mathbf{P}^{(h-1)}$.

\State
\textbf{if} $(F_2 < F_1)$ \textbf{then}
$\mathbf{P}_{j_{\text{best}}}^{(h)} \gets \mathbf{q}_{j_{\text{best}}, v_{\text{best}}}$.

\If{$(h > 1)$}

\State 
\textbf{if} $(|F_1-F_2| < \tau_1$ and $\eta>\phi)$ \textbf{then} 
$\eta \gets \eta/\rho$.

\EndIf

\State
$\eta^{(h)} \gets \eta$ and $h \gets h+1$.

\EndWhile

\State
$\widehat{\mathbf{P}}^{(R)} \gets \mathbf{P}^{(h)}$. 

\If{$|f(\widehat{\mathbf{P}}^{(R)}) - f(\widehat{\mathbf{P}}^{(R-1)})| < \tau_2$}

\State
\Return $\widehat{\mathbf{P}} = \widehat{\mathbf{P}}^{(R)}$ as \textbf{final solution}.

\State
\textbf{exit}
\Else

\State
$R \gets R+1$.

\State
\textbf{goto} \emph{top}.
\EndIf
\end{algorithmic}
\end{algorithm}

\begin{figure}[ht]
\centering
\includegraphics[width=1\textwidth]{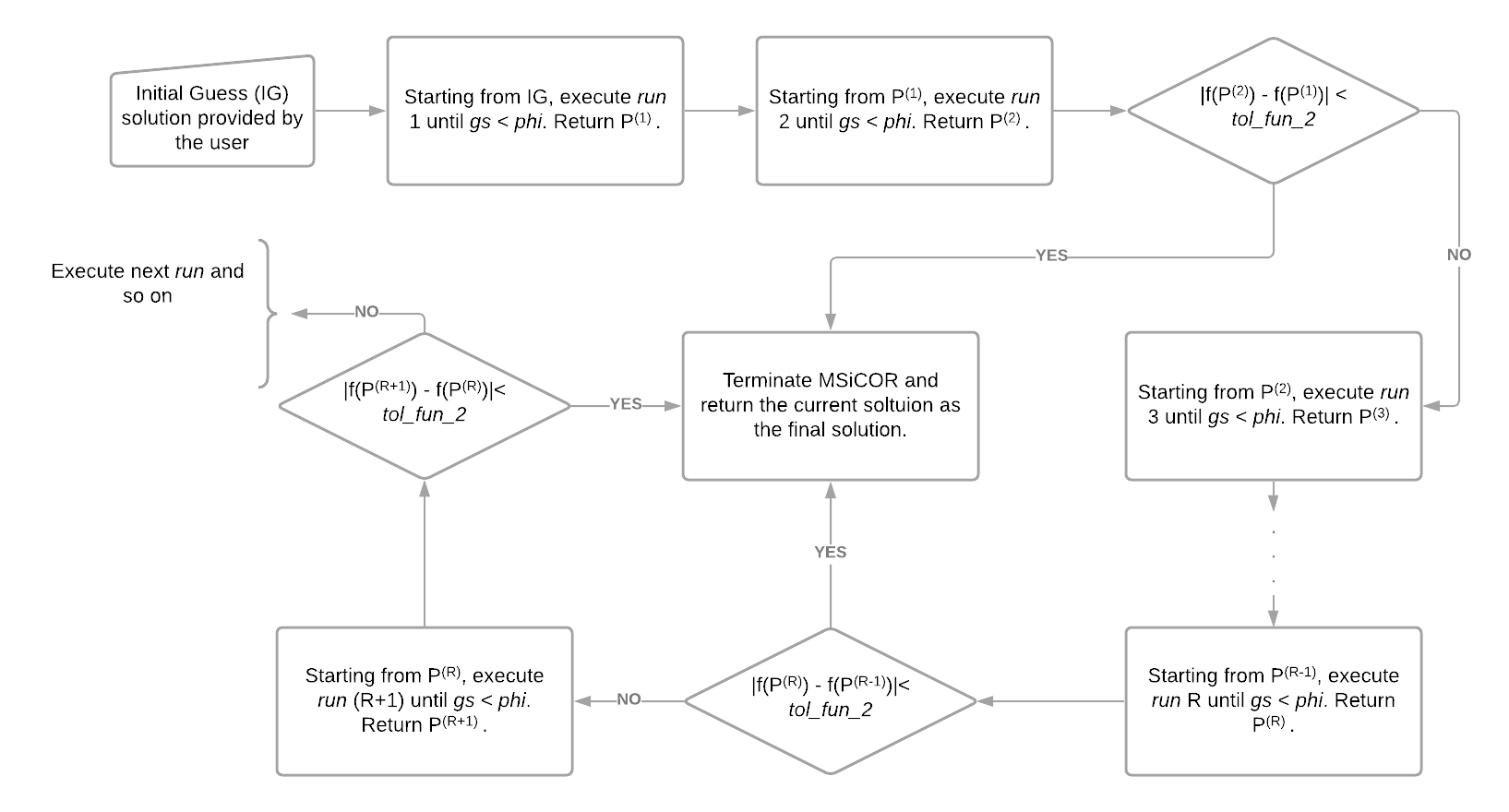}
\caption{Flowchart of the MSiCOR : \textit{gs} (same as $\eta$) denotes the global step-size, $phi\, (\phi)$ denotes the step-size threshold, $P^{(h)}$ denotes the solution returned at the end of $h$-th run, \textit{tol\_fun\_2} (same as $\tau_2$) is a cut-off value.}
\label{flowchart}
\end{figure}

\subsection{Comparative performance using benchmark functions}
\label{sec.benchmark}
In order to evaluate the comparative performance of MSiCOR, we consider the minimization problem of four benchmark functions \citep{Jamil2013} namely Rastrigin function, Ackley's function, sphere function and Griewank function whose parameter spaces are modified to be a collection of simplexes (see Section C of the supplementary material for function expressions). The benchmark functions are quite challenging to minimize because of their typical complex structure along with presence of multiple minimums (generally) which pose challenges in finding the global minimum. It is impossible to design an optimization algorithm which can always find the global minimum for any black-box function. Although some global optimization algorithms are backed up with theoretical evidence on global convergence (e.g., \cite{Schmitt2001} on genetic algorithm), however, those convergence properties are derived based on a lot of assumptions regarding the properties of the objective function. Firstly those properties might not hold true in practice. Secondly there is no way to verify if those properties hold true for a black-box function since the function is already `black-box' by nature. Therefore, to comprehend the performance of any heuristic, global or black-box optimization technique, a common practice is to conduct experiments based on minimizing benchmark functions; and eventually based on average computation time and quality of the obtained solutions, one may get an idea about the relative performance of an algorithm. In this context, along with MSiCOR we also consider Genetic Algorithm (GA), Sequential Quadratic Programming (SQP) and Interior Point (IP) optimization techniques to minimize the aforementioned benchmark functions (see Section C of the supplementary material for details). IP and SQP algorithms search for local minimum while optimizing any function and in general they are less time consuming. On the other hand GA tries to find global minimum being more time consuming. These above-mentioned well-known algorithms are available in Matlab via the Optimization Toolbox functions \texttt{fmincon} (for IP and SQP algorithm)  and \texttt{ga} (for GA). While using IP and SQP algorithms, the upper bound for maximum number of iterations and function evaluations have been set to be infinity each. For GA, the default options of `ga' function in Matlab has been considered. MSiCOR is implemented in Matlab. The comparative study has been performed for the cases $n=5, d=5$ and $n=5,d=10$ for all the above-mentioned algorithms, where $n$ is the number of simplexes and $(d-1)$ is the dimension of each simplex blocks. All the computations have been performed in a computer with 64-Bit Windows 8.1, Intel i7 3.6GHz processor, 32GB RAM. For each case, all the algorithms have been initialized from 100 randomly generated starting points. The true minimum of all considered modified benchmark functions is 0. The average computation time (in seconds, also standard deviation in parenthesis), obtained minimum value of the objective functions and standard deviations of obtained objective function values for each cases have been noted down in \cref{dd}. It is observed that MSiCOR provides better solution for considered benchmark functions compared to other methods, except for modified sphere function. The sphere function being convex, it is expected for IP and SQP to perform well while optimizing it. But MSiCOR outperforms GA for all considered benchmark functions. Although MSiCOR took more time on average to converge compared to IP and SQP, it took much less time compared to GA. It is also noted that the standard deviation of the objective function values at the obtained solutions is the least for MSiCOR, which implies that MSiCOR algorithm is more robust and less dependent on starting points compared to other algorithms. To check the performance of the proposed algorithm in higher dimensional problems, additional simulation studies have been performed. In \cref{tab_high_dim} we explore the performance of MSiCOR for higher dimensional simplex blocks ($n = 20,50$, $d = 10,20$).

\begin{table}[]
\centering
\resizebox{1\columnwidth}{!}{
\begin{tabular}{|c|c|ccc|ccc|}
\hline
\multirow{2}{*}{Functions} & \multirow{2}{*}{Algorithms} & \multicolumn{3}{c|}{$n=5, d=5$} & \multicolumn{3}{c|}{$n=5,d=10$} \\ \cline{3-8} 
 &  & \multicolumn{1}{c|}{min. value} & \multicolumn{1}{c|}{sd of solutions} & mean time (sd) & \multicolumn{1}{c|}{min. value} & \multicolumn{1}{c|}{sd of solutions} & mean time (sd) \\ \hline
\multirow{4}{*}{\begin{tabular}[c]{@{}c@{}}Modified\\ Rastrigin\end{tabular}} & MSiCOR & \multicolumn{1}{c|}{1.39e + 01} & \multicolumn{1}{c|}{8.20} & 6.83 (2.17) & \multicolumn{1}{c|}{2.30e + 01} & \multicolumn{1}{c|}{19.59} & 9.12 (3.26) \\ \cline{2-8} 
 & IP & \multicolumn{1}{c|}{4.01e + 02} & \multicolumn{1}{c|}{276.27} & 0.12 (0.04) & \multicolumn{1}{c|}{9.14e + 02} & \multicolumn{1}{c|}{666.64} & 4.35 (17.09) \\ \cline{2-8} 
 & SQP & \multicolumn{1}{c|}{1.62e + 02} & \multicolumn{1}{c|}{150.51} & 0.05 (0.01) & \multicolumn{1}{c|}{8.78e + 02} & \multicolumn{1}{c|}{244.09} & 0.21 (0.02) \\ \cline{2-8} 
 & GA & \multicolumn{1}{c|}{1.51e + 01} & \multicolumn{1}{c|}{56.72} & 20.26 (2.26) & \multicolumn{1}{c|}{2.32e + 02} & \multicolumn{1}{c|}{1157.85} & 22.08 (4.40) \\ \hline
\multirow{4}{*}{\begin{tabular}[c]{@{}c@{}}Modified\\ Ackley's\end{tabular}} & MSiCOR & \multicolumn{1}{c|}{3.48e - 02} & \multicolumn{1}{c|}{1.06} & 6.50 (2.53) & \multicolumn{1}{c|}{6.51e - 02} & \multicolumn{1}{c|}{0.23} & 9.81 (3.62) \\ \cline{2-8} 
 & IP & \multicolumn{1}{c|}{4.86e + 01} & \multicolumn{1}{c|}{6.87} & 0.13 (0.09) & \multicolumn{1}{c|}{6.86e + 01} & \multicolumn{1}{c|}{4.50} & 0.36 (0.07) \\ \cline{2-8} 
 & SQP & \multicolumn{1}{c|}{7.75e + 00} & \multicolumn{1}{c|}{14.62} & 0.06 (0.01) & \multicolumn{1}{c|}{2.17e + 01} & \multicolumn{1}{c|}{14.77} & 0.18 (0.04) \\ \cline{2-8} 
 & GA & \multicolumn{1}{c|}{6.63e + 00} & \multicolumn{1}{c|}{11.94} & 19.93 (2.60) & \multicolumn{1}{c|}{1.60e + 01} & \multicolumn{1}{c|}{7.32} & 24.33 (3.26) \\ \hline
\multirow{4}{*}{\begin{tabular}[c]{@{}c@{}}Modified\\ Sphere\end{tabular}} & MSiCOR & \multicolumn{1}{c|}{7.42e - 05} & \multicolumn{1}{c|}{\textless 1e - 06} & 5.86 (0.59) & \multicolumn{1}{c|}{5.13e - 04} & \multicolumn{1}{c|}{5.7e - 06} & 9.92 (1.37) \\ \cline{2-8} 
 & IP & \multicolumn{1}{c|}{1.63e - 15} & \multicolumn{1}{c|}{\textless 1e - 06} & 0.07 (0.03) & \multicolumn{1}{c|}{8.90e - 16} & \multicolumn{1}{c|}{\textless 1e - 06} & 0.12 (0.04) \\ \cline{2-8} 
 & SQP & \multicolumn{1}{c|}{5.60e - 12} & \multicolumn{1}{c|}{\textless 1e - 06} & 0.04 (0.00) & \multicolumn{1}{c|}{2.72e - 11} & \multicolumn{1}{c|}{\textless 1e - 06} & 0.17 (0.02) \\ \cline{2-8} 
 & GA & \multicolumn{1}{c|}{5.80e - 01} & \multicolumn{1}{c|}{45.25} & 20.25 (1.89) & \multicolumn{1}{c|}{1.34e + 02} & \multicolumn{1}{c|}{633.55} & 22.76 (5.43) \\ \hline
\multirow{4}{*}{\begin{tabular}[c]{@{}c@{}}Modified\\ Griewank\end{tabular}} & MSiCOR & \multicolumn{1}{c|}{4.94e - 01} & \multicolumn{1}{c|}{0.40} & 4.29 (1.42) & \multicolumn{1}{c|}{1.08e + 00} & \multicolumn{1}{c|}{0.55} & 7.75 (1.46) \\ \cline{2-8} 
 & IP & \multicolumn{1}{c|}{1.07e + 01} & \multicolumn{1}{c|}{28.54} & 0.17 (0.05) & \multicolumn{1}{c|}{2.19e - 01} & \multicolumn{1}{c|}{10.74} & 0.31 (0.07) \\ \cline{2-8} 
 & SQP & \multicolumn{1}{c|}{1.44e + 01} & \multicolumn{1}{c|}{13.20} & 0.07 (0.01) & \multicolumn{1}{c|}{6.97e - 01} & \multicolumn{1}{c|}{0.99} & 0.28 (0.03) \\ \cline{2-8} 
 & GA & \multicolumn{1}{c|}{1.04e + 01} & \multicolumn{1}{c|}{3.57} & 18.53 (1.44) & \multicolumn{1}{c|}{2.73e + 02} & \multicolumn{1}{c|}{12.39} & 26.67 (3.16) \\ \hline
\end{tabular}}
\caption{Comparative study of MSiCOR, IP, SQP and GA for optimizing modified Rastrigin, Ackley, Sphere and Griewank function starting from 100 randomly generated initial points. Objective function values at the best (or minimum) solution out of 100 obtained solutions using the aforementioned algorithms are noted down for each case (under `min. value' column); along with standard deviation (s.d.) of objective function values at the solutions (under `sd of solutions' column). Also mean and s.d. (within parenthesis) of computation times (in seconds) are noted for each case (under `mean time (sd)' column).}
\label{dd}
\end{table}

\begin{table}[]
\centering
\resizebox{1\columnwidth}{!}{%
\bgroup
\begin{tabular}{|c|c|c|c|c|c|}
\hline
Functions & Measures & $n = 20, d = 10$ & $n = 20, d = 20$ & $n = 50, d = 10$ & $n = 50, d = 20$ \\ \hline
\multirow{3}{*}{\begin{tabular}[c]{@{}c@{}}Modified \\ Rastrigin\end{tabular}} & min. value & 1.03e - 02 & 2.04e - 02 & 2.58e - 02 & 5.07e - 02 \\ \cline{2-6} 
 & sd of solutions & 0.41 & 0.32 & 0.51 & 0.39 \\ \cline{2-6} 
 & mean time (sd) & 53.28 (10.05) & 80.27 (25.21) & 438.23 (78.37) & 689.45 (129.72) \\ \hline
\multirow{3}{*}{\begin{tabular}[c]{@{}c@{}}Modified\\ Ackley's\end{tabular}} & min. value & 4.01e - 02 & 4.01e - 02 & 1.00e - 01 & 1.00e - 01 \\ \cline{2-6} 
 & sd of solutions & 0.92 & 0.56 & 0.96 & 1.02 \\ \cline{2-6} 
 & mean time (sd) & 110.94 (11.29) & 150.83 (14.85) & 611.41 (49.88) & 1348.45 (131.55) \\ \hline
\multirow{3}{*}{\begin{tabular}[c]{@{}c@{}}Modified \\ Sphere\end{tabular}} & min. value & 5.24e - 05 & 1.02e - 04 & 1.30e - 04 & 2.59e - 04 \\ \cline{2-6} 
 & sd of solutions & \textless 0.01 & \textless 0.01 & \textless 0.01 & \textless 0.01 \\ \cline{2-6} 
 & mean time (sd) & 41.99 (2.04) & 87.83 (5.72) & 308.74 (10.65) & 577.34 (17.90) \\ \hline
\multirow{3}{*}{\begin{tabular}[c]{@{}c@{}}Modified \\ Griewank\end{tabular}} & min. value & 1.04e - 01 & 1.27e - 01 & 9.25e - 02 & 2.43e - 01 \\ \cline{2-6} 
 & sd of solutions & 0.58 & 0.28 & 1.15 & 0.79 \\ \cline{2-6} 
 & mean time (sd) & 58.48 (2.98) & 201.70 (10.59) & 741.82 (20.01) & 3291.46 (210.35) \\ \hline
\end{tabular}
\egroup}
\caption{Performance of MSiCOR in higher dimensional modified Rastrigin, Ackley, Sphere and Griewank function starting from 100 randomly generated initial points. Objective function values at the best (or minimum) solution out of 100 obtained solutions using MSiCOR are noted down for each case (`min. value' row); along with standard deviation (s.d.) of objective function values at the solutions (`sd of solutions' row). Also mean and s.d. (within parenthesis) of computation times (in seconds) are noted for each case (`mean time (sd)' row).}
\label{tab_high_dim}
\end{table}

\subsection{Theoretical properties}
\label{sec_theory}
The greatest challenge of solving a non-convex optimization problem is that in general algorithms cannot be designed to guarantee reaching the global optima. However, it is a desirable property of any algorithm that it should reach a global minimum when the function is convex. In this section, via \cref{theorem}, it is shown that taking the values of the parameters $\phi$, $\tau_1$ and $\tau_2$ significantly small, the stopping criteria of the proposed algorithm ensures that the solution obtained is a global minimum in case the objective function is convex. The detailed proof of \cref{theorem} is provided in Section B of the Supplementary material.

\begin{theorem}
\label{theorem}
Suppose $\bs=\Delta^{n_1-1} \times \cdots \times \Delta^{n_B-1}$ and $f$ is convex, continuous and differentiable on $\bs$. Suppose $\mathbf{u} = (\mathbf{u}_1,\dots, \mathbf{u}_B) \in \bs$ and $\mathbf{u}_j = (u_{j,1}, \dots, u_{j,n_j}) \in \Delta^{n_j-1}$ for $j = 1,\dots, B$, and each of its coordinates are non-zero. Consider a sequence $\delta_{j,k} = s_j/\rho^k$ for $k\in \mathbb{N}$, $s_j > 0$, $\rho > 1$ for all $j = 1,\dots, B$. Define $\mathbf{u}_{j,k}^{(i+)} = (u_{j,1} - \delta_{j,k}/(n-1), \dots, u_{j,i-1} - \delta_{j,k}/(n-1), u_{j,i} +\delta_{j,k}, u_{j,i+1} - \delta_{j,k}/(n-1), \dots, u_{j,n_j} - \delta_{j,k}/(n-1))$ and $\mathbf{u}_{j,k}^{(i-)} = (u_{j,1} + \delta_{j,k}/(n-1), \dots, u_{j,i-1} + \delta_{j,k}/(n-1), u_{j,i} -\delta_{j,k}, u_{j,i+1} + \delta_{j,k}/(n-1), \dots, u_{j,n_j} + \delta_{j,k}/(n-1))$ for $j = 1,\dots, B$ and $i = 1,\dots, n_j$. If for all $k\in \mathbb{N}$, $f(\mathbf{u}) \leq f(\mathbf{u}_1, \dots, \mathbf{u}_{j-1}, \mathbf{u}_{j,k}^{(i+)}, \mathbf{u}_{j+1}, \dots, \mathbf{u}_{B})$ and $f(\mathbf{u}) \leq f(\mathbf{u}_1, \dots, \mathbf{u}_{j-1}, \mathbf{u}_{j,k}^{(i-)}, \mathbf{u}_{j+1}, \dots, \mathbf{u}_{B})$ (whenever $\mathbf{u}_{j,k}^{(i+)}, \mathbf{u}_{j,k}^{(i-)} \in \Delta^{n_k-1})$ for $j = 1,\dots, B$ and $i = 1,\dots, n_j$, $\mathbf{u}$ is a point of global minimum of $f$. \\
\end{theorem}

It should be noted that taking step-size threshold $\phi$ small enough, the allowable values of local step-sizes $s_{j,i}^{+}$ and $s_{j,i}^{-}$ can be taken as close to zero as required. Also note that in Theorem \ref{theorem}, the role of $\delta_{j,k}$ is analogous to that of $s_{j,i}^{+}$ and $s_{j,i}^{-}$ in \cref{sec.algo}. In other words, in the proposed algorithm if we take $\tau_1 = 0$ and $M_{\text{iter}} = \infty$, the iterations within a run stops when for very small value of $s_{j,i}^{+}$ and $s_{j,i}^{-}$ for $j=1,\dots, B$ and $i=1,\dots, n_j$, corresponding movements in the neighborhood do not yield better solution than the current solution. Hence, it that scenario, the obtained solution by the proposed algorithm is a global minimum if the objective function follows the regularity conditions considered in Theorem \ref{theorem}. Note that, for a convex function satisfying the regularity conditions, the convergence criteria ensures that at the end of any run the solution obtained is a global minimum. Hence, in this case, evaluation of only one run will be enough to find the global minimum.

\section{Simulation study}
\label{sec_sim}

\subsection{Setup}

We consider $K=1000$ patients belonging to $L = 3$ clusters. The number of states is chosen to be $n=10$. For each cluster, we simulate rows of each $10 \times 10$ transition matrix $M_l$ from an uniform Dirichlet distribution. Initial state distribution vectors corresponding to clusters $s_1,\dots,s_L$ are also generated from an uniform Dirichlet distribution. 
We choose $p=4$ covariates (including an intercept term), which are randomly generated from a $\mathbf{N}(0,I_4)$ distribution.
For each patient, we generate the length of their observed Markov chain (i.e., $h_i$ for the $i$-th patient) from a discrete uniform distribution on $\{6,\dots,12\}$. The length of the chain of the same subject may vary across different simulations. 
To generate the observations, we first compute the probabilities of each patient belonging to the different clusters $w_{kl}$ using \cref{eq.prob2}. Then, using multinomial draws corresponding to prior probabilities, cluster memberships of each patient is evaluated. For each patient, based on corresponding cluster membership, initial state is generated using multinomial draws from initial state distribution vector. The following states of the Markov chain are generated using multinomial draws from corresponding cluster-specific transition matrix.

\subsection{Parameter estimation}\label{param_est}

In order to estimate the parameters, we first look for a warm starting point for cluster-specific initial state distribution vectors and transition matrices. We first maximize
$L_{\bM,\bs,\Gamma}(Y_1,\dots,Y_K)$ given by \cref{eq.prob2} taking $w_{kl} = \frac{1}{L}$, (i.e., equal weight for each cluster) using MSiCOR and obtain initial estimates for $\bM$ and $\bs$. Then for given initial values of $\bM$ and $\bs$, we maximize $L_{\bM,\bs,\Gamma} (Y_1,\dots,Y_K)$ given by \cref{eq.prob2} and thus we obtain an initial estimate for $\Gamma$. For maximizing the likelihood as a function of $\Gamma$ for given $(\bM,\bs)$, we use \texttt{patternsearch} function in MATLAB which performs a global maximization of the likelihood as a function of $\Gamma$. Once we obtain the initial estimates for $(\bM,\bs,\Gamma)$, we then maximize the likelihood updating $(\bM,\bs,\Gamma)$ simultaneously within each iteration of the proposed algorithm. Note that the parameters $\bM,\bs$ are a collection of simplexes and hence can be directly estimated using MSiCOR, however, $\Gamma$ being unconstrained, we adopt another optimization technique for updating $\Gamma$ at each iteration within MSiCOR. \cite{Das2023b} proposed global optimization technique Recursive Modified Pattern Search (RMPS) for optimizing hyper-rectangular parameter space. RMPS can be easily modified for optimizing unconstrained parameter space (see Algorithm 1 in the supplementary file). Using MSiCOR and modified RMPS for unconstrained optimization, at each iteration, we update $(\bM,\bs,\Gamma)$ simultaneously. This joint algorithm is provided in Algorithm \ref{joint_algo}. Here in the algorithm, $(\bM,\bs)$ being collection of simplexes, we denote it by $\mathbf{P}$ and unconstrained parameter $\Gamma$ is denoted by $\mathbf{l}$.

\begin{algorithm}
\caption{Algorithm for jointly updating multiple simplexes (using MSiCOR) and unconstrained parameter vector (using RMPS).}\label{joint_algo}
\begin{algorithmic}[1]
\State $R \gets 1$
\BState \emph{top}:
\State $h \gets 1$
\State $\eta^{(0)},\eta^{(1)} \gets s_{\text{initial}}$
\If{$R = 1$}
\State $\boldsymbol{\mathbf{P}}^{(0)} \gets \text{Initial guess}$ (Cell/list of simplexes of dimensions $n_1, n_2, \dots, n_B$ respectively)
\State $\boldsymbol{\mathbf{l}}^{(0)} \gets \text{Initial guess}$ (unconstrained parameter vector)
\Else 
\State $\boldsymbol{\mathbf{P}}^{(0)} \gets \widehat{\mathbf{P}}^{(R-1)}$
\State $\boldsymbol{\mathbf{l}}^{(0)} \gets \widehat{\mathbf{l}}^{(R-1)}$
\EndIf
\While{$(h \leq M_{\text{iter}}$ and $\eta^{(h)} > \phi)$}
\State $F \gets f(\mathbf{P}^{(h-1)})$
\State $\eta \gets \eta^{(h-1)}$
\State Find $f_{j_{\text{best}},v_{\text{best}}}$ and set $F_1 \gets f_{j_{\text{best}},v_{\text{best}}}$ using steps 12-36 of Algorithm \ref{euclid}.
\State Find $g_{k_{\text{best}}}$ and set $F_2 \gets g_{k_{\text{best}}}$ using steps 12-19 of Algorithm 1 of the supplementary file.
\State $\mathbf{P}^{(h)} \gets \mathbf{P}^{(h-1)}$
\State $\mathbf{l}^{(h)} \gets \mathbf{l}^{(h-1)}$
\State $F' \gets \min(F_1,F_2)$
\If{$(F' < F)$}
\If{$(F_1 < F_2)$}
\State $\mathbf{P}_{j_{\text{best}}}^{(h)} \gets \mathbf{q}_{j_{\text{best}}, v_{\text{best}}}$
\Else 
\State $\mathbf{l}^{(h)} \gets \mathbf{l}_{k_{\text{best}}}$
\EndIf
\EndIf
\If{$(h > 1)$}
\If{$(|F-\min(F,F')| < tol\_fun$ and $\eta>\phi)$}
\State $\eta \gets \eta/\rho$
\EndIf
\EndIf
\State $\eta^{(h)} \gets \eta$
\State $h \gets h+1$
\EndWhile
\State $(\widehat{\mathbf{P}}^{(R)}, \widehat{\mathbf{l}}^{(R)}) \gets (\mathbf{P}^{(h)},\mathbf{l}^{(h)})$, 
\If{$|f(\widehat{\mathbf{P}}^{(R)},\widehat{\mathbf{l}}^{(R)}) - f(\widehat{\mathbf{P}}^{(R-1)},\widehat{\mathbf{l}}^{(R-1)})| < \tau_2$ }
\State \Return $\widehat{\mathbf{P}} = \widehat{\mathbf{P}}^{(R)}, \widehat{\mathbf{l}} = \widehat{\mathbf{l}}^{(R)}$ as \textbf{final solution}
\State \textbf{exit}
\Else
\State $R \gets R+1$
\State \textbf{goto} \emph{top}.
\EndIf
\end{algorithmic}
\end{algorithm}

\subsection{Mapping true clusters to estimated clusters}
One of the main concerns regarding parameter estimation in mixture models or mixture of experts (MoE) \citep{Masoudnia2014} is to ensure identifiability. MSiCOR (for MMM without covariates) and MSiCOR-RMPS hybrid algorithm (for MMM with covariates) are both designed to find the point of maxima of the mixture likelihood; subsequently we obtain the estimated parameter values corresponding to each cluster. In case we perform MMM clustering to any new dataset, we can divide the sample population in desired number of clusters and identifiability might not be a concern in absence of reference clusters. However, to assess the estimation performance based on simulation study, it is crucial to map the estimated clusters to the true reference clusters. For any given cluster, let us denote the initial state distribution vector and the transition matrices by $s_{n \times 1}$ and $M_{n\times n}$ respectively. From $s$ and $M$, we construct the appended matrix $A = [s; M]_{(n+1)\times n}$. Given two appended matrices $B$ and $C$, we define the total variation distance (TVD) as
\begin{align*}
\mathbf{TVD}(B,C)
& =
\sum_{i=1}^{n+1}\sum_{j=1}^n|B(i,j) - C(i,j)|.
\end{align*}
Suppose we denote the true appended matrices (constructed using the initial state distribution vector and the transition matrices) by $A_1,A_2,\dots,A_L$ and the estimated appended matrices by $B_1,\dots,B_L$. Then we find the permutation $\sigma$ of $\{1,\dots,L\}$ for which $\sum_{l=1}^L \mathbf{TV}(A_l,B_{\sigma(l)})$ is minimized. Thus corresponding permutation helps us identifying the mapping across the true clusters to the estimated clusters.
\subsection{Results}
After we obtain the estimates of $(\bM,\bs,\Gamma)$, using \cref{eq.memebership}, we find the membership probabilities of each subject to three estimated clusters. We consider each subject to belong to the cluster with corresponding highest membership probabilities. To compare the true and the estimated cluster membership of the patients, first we map three true clusters to three estimated clusters. Then we calculate the true positive rate (TPR) of membership of the subjects. In order to calculate the standard error of the parameter estimates, we perform 20 simulation iterations. The true and estimated parameter values are provided in \cref{true_est_prop1}. It is noted that the true and the estimated cluster covariate coefficients are close.

\begin{table}[ht]
\centering

\begin{tabular}{|c|c|c|c|c|c|}
\hline
Clusters & & Intercept & Var 1 & Var 2 & Var 3 \\ \hline
\multirow{2}{*}{Cluster 2} & True & $-0.75$ & $0.64$ & $-0.19$ & $-0.13$ \\ \cline{2-6} 
& Est. & $-0.738$ $(0.0212)$ & $0.614$ $(0.0229)$ & $-0.185$ $(0.0239)$ & $-0.149$ $(0.0162)$ \\ \cline{2-6}
\multirow{2}{*}{Cluster 3} & True & $-0.07$ & $1.12$ & $-0.25$ & $-0.4$ \\ \cline{2-6} 
& Est. & $-0.086$ $(0.0199)$ & $1.123$ $(0.0192)$ & $-0.270$ $(0.0210)$ & $-0.414$ $(0.0213)$ \\ \hline
\end{tabular}
\caption{True and estimated coefficients of the covariates in the simulation study. Standard errors are provided in the parenthesis.}
\label{true_est_prop1}
\end{table}
As mentioned earlier, based on posterior membership probabilities from estimated parameters, the membership of the subjects can also be estimated. In \cref{true_est_prop2} we note down the true and estimated proportions of subjects belonging to each cluster.
\begin{table}[H]
\centering

\begin{tabular}{|c|c|c|c|}
\hline
& Cluster 1 & Cluster 2 & Cluster 3 \\ \hline
True proportion & $0.422$ & $0.174$ & $0.404$ \\ \hline
Estimated proportion & $0.422$ $(3.4 \times 10^{-3})$ & $0.178$ $(3.2 \times 10^{-3})$ & $0.400$ $(3.9 \times 10^{-3})$ \\ \hline
\end{tabular}
\caption{True and estimated cluster proportions (empirical standard errors shown in parentheses).}
\label{true_est_prop2}
\end{table}

\subsection{Comparison with the EM algorithm}
\label{sec_EM}
In order to compare the performance of MSiCOR with the EM algorithm in the MMM (without covariates), we consider another simulation study. Here again we consider $K=1000$ patients where each patient belongs to one of the $L=3$ clusters. The dimension of the state-space is taken to be $n = 10$. Corresponding to each cluster, initial state distribution vectors are generated from uniform Dirichlet distribution. The transition matrices corresponding to the first, second and the third clusters are taken to be $30\%$, $50\%$ and $70\%$ sparse respectively. The non-zero elements of each row of the transition matrices corresponding to each cluster are generated from uniform Dirichlet distribution. For each patient, the length of their observed Markov chain (i.e., $h_i$ for the $i$-th patient) is chosen following a discrete uniform distribution on $\{6,\dots,12\}$. Then based on the obtained true transition matrices and the initial state probability vectors, the patient-specific state-space sequences are generated. To fit EM, we use the \texttt{fit\_model} function in \texttt{seqHMM} R package \citep{Helske2019}. The simulation experiment is repeated 50 times and corresponding TVD is measured for both MSiCOR and EM methods. Using TVD, estimated clusters are mapped to the true clusters. After mapping, the misclassification rate (MR) is calculated. In \cref{em_table} it is observed that MSiCOR performs better than the EM algorithm in terms of TVD and MR measures. It is noted that MSiCOR performs better than the EM algorithm yielding lower TVD and MR compared to the EM algorithm. It is also noted that unlike MSiCOR, the EM algorithm seems not to converge on a few occasions yielding very high TPR and MR values (\cref{TVD}).

\begin{table}[H]
\centering
\begin{tabular}{|c|c|c|c|c|c|}
\hline
Sample size & Method & TVD & Max TVD & MR & Max MR \\ \hline
\multirow{2}{*}{$n = 500$} & MSiCOR & $4.57$ $(0.13)$ & $5.27$ & $0.000$ $(0.000)$ & $0.004$ \\
& EM & $5.41$ $(0.66)$ & $26.42$ & $0.019$ $(0.013)$ & $0.474$\\ \hline
\multirow{2}{*}{$n = 1000$} & MSiCOR & $3.18$ $(0.04)$ & $3.63$ & $0.000$ $(0.000)$ & $0.003$ \\
& EM & $3.66$ $(0.48)$ & $26.34$ & $0.010$ $(0.010)$ & $0.434$ \\ \hline
\end{tabular}
\caption{Comparison of MSiCOR and EM based on total variation distance (TVD) and misclassification Rate (MR) based on 50 simulation experiments; mean TVD and TPR values are noted down in the table. Standard errors are provided in parentheses.}
\label{em_table}
\end{table}

\begin{figure}[ht]
\centering
\includegraphics[width=0.8\textwidth]{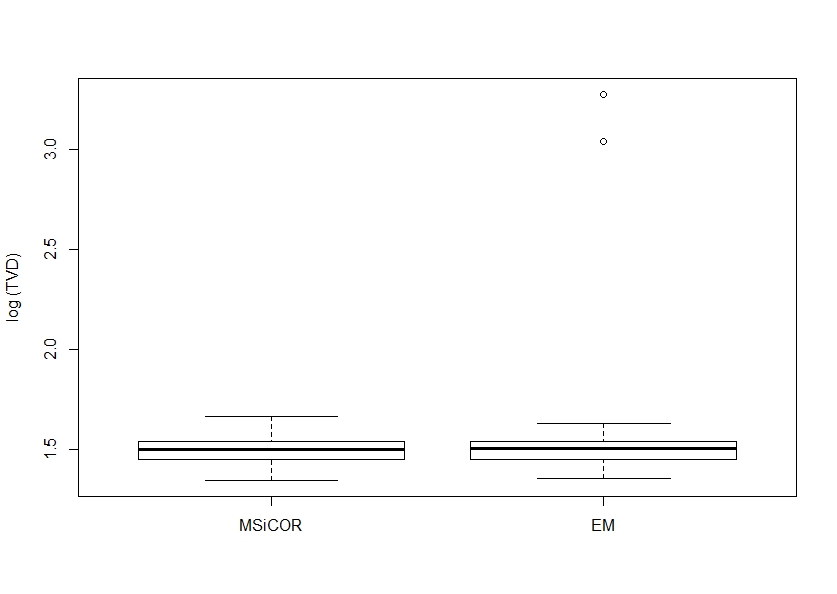}
\caption{Boxplot of $\log$ of total variation distance (TVD) for MSiCOR and EM based on 50 simulation experiments with sample size 500.}
\label{TVD}
\end{figure}

\section{Application to clustering of medication-sequence data of Multiple Sclerosis patients in EHR cohort}
\label{sec_real}

As a test case, we use MS disease-modifying therapy (DMT) sequence data from an electronic health record (EHR) cohort based at the Massachusetts General and Brigham hospital system (Boston, US) which includes the Comprehensive Longitudinal Investigation of Multiple Sclerosis at Brigham and Women's Hospital (CLIMB) cohort \citep{Zhang2020, Ahuja2021,Hou2021,Liang2022,Xia2013}. The EHR cohort contains patient-level data, including DMTs as well as a number of clinical and demographic variables. In this data set, there are twelve available DMTs for MS patients: alemtuzumab, cyclophosphamide, daclizumab, dimethyl fumarate, fingolimod, glatiramer acetate, interferon-beta, mitoxantrone, natalizumab, ocrelizumab, rituximab and teriflunomide. Of these, daclizumab has been withdrawn from the market and only a few patients received this DMT. Therefore we exclude daclizumab from our analysis by omitting the corresponding encounters from the MS DMT sequences. We combine rituximab and ocrelizumab under the same mechanistic category (that is, B-cell Depletion). We thus have ten DMT categories in total, which form the state-space of our Markov model. Further, we only consider patients who started on MS DMTs on or after January 1, 2006, because the Mass General Brigham system began implementation of the electronic prescriptions during 2005. 

To avoid over-counting given that consecutive visits that are only a few days apart can sometimes list the same DMT prescription, we combine the MS DMT observations into three-month period clusters, starting from DMT start date. Within any three-month period, the consecutive same DMTs are counted as one observation. If a patient takes one DMT during a given three-month period, we count it as one observation representing the given three-month period. For example, during any 3 month period, if the encounters are $A \to A \to A \to A \to A \to A \to A$ or $A \to A \to A \to A$ or $A$ (as long as all encounters in that three-month period are the same DMT $A)$, we take the observation as only $A$ for that 3 month period. On the other hand, consider a scenario where a patient has been on DMT sequence $A \to A \to A \to A \to B \to B \to B \to A \to A \to A \to C \to C$ during a three-month period. By including the unique consecutive DMTs into one observation, we then get $A \to B \to A \to C$ as the representative observation for that three-month period. 

\begin{table}[ht]
\centering

\begin{tabular}{|c|l|}
\hline
Number of clusters & \multicolumn{1}{c|}{BIC} \\ \hline
3 & \textbf{13072} \\ \hline
4 & 13556 \\ \hline
5 & 14237 \\ \hline
6 & 15065 \\ \hline
\end{tabular}
\caption{Summary of Bayesian Information Criterion (BIC) values after fitting mixture Markov model with covariates for number of clusters $L = 3,\dots,6$.}
\label{BIC}
\end{table}

In the EHR cohort, for clustering, we only consider the patients for which clinical and demographic data are available along with MS DMT sequnce data; also those patients must start on MS DMT on or after 2006. After applying the aforementioned filters, we finally cluster 822 patients. In the mixture Markov model analysis, we consider the following covariates that are routine in MS research: age at diagnosis, disease duration, gender, race (white, black, and others). The disease duration is the time elapsed from the year of first neurological symptom to the DMT start year. The parameters are estimated using the joint algorithm (using MSiCOR and modified RMPS for unconstrained optimization) as described in \cref{param_est}.

\begin{figure}[ht]
\centering
\includegraphics[width=0.95\textwidth]{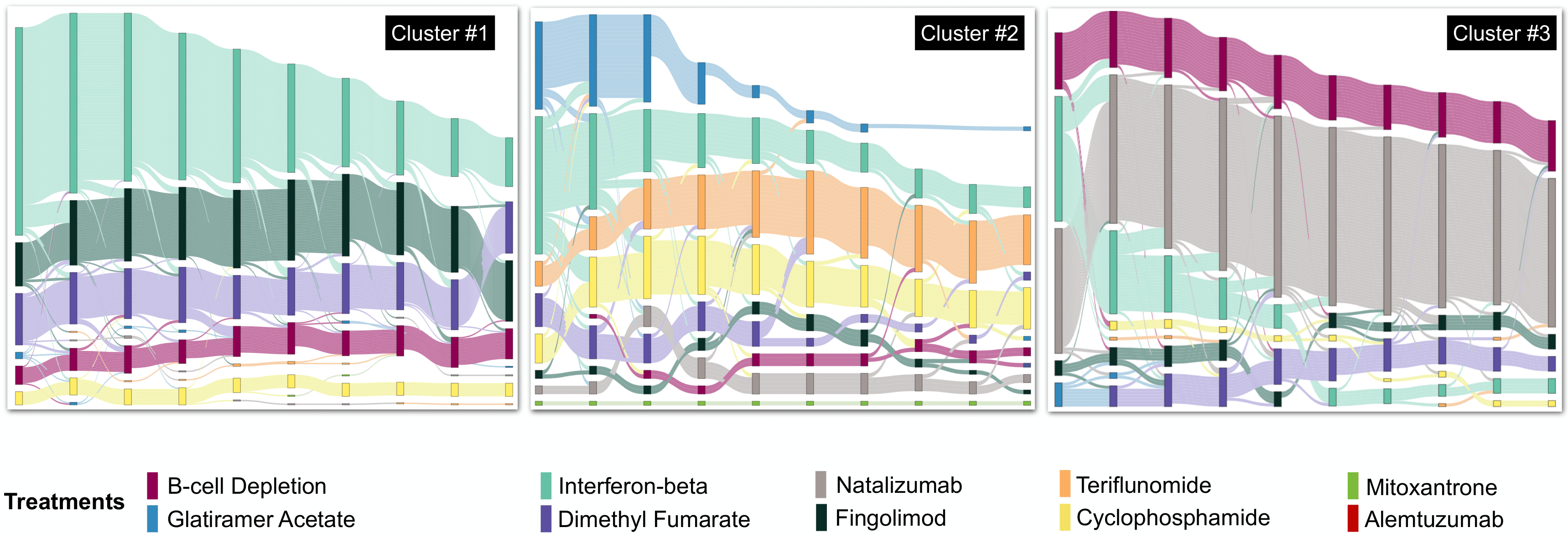}
\caption{Upto first 10 MS DMTs are plotted for the patients from each cluster. B-cell Depletion DMT group consists of Rituximab and Ocrelizumab.}
\label{dmt_plot}
\end{figure}

To identify the number of clusters, we fit the MMM with covariates for number of clusters $3,\dots,6$. Using Bayesian Information Criterion (BIC) \citep{Schwarz1978}, we identify the optimal number of clusters to be 3 (\cref{BIC}). After we cluster the 822 patients into 3 DMT sequence clusters, we identify the corresponding cluster for any patient by calculating the posterior cluster membership probabilities given by \cref{eq.memebership}. Once the cluster membership probabilities are calculated for a patient, we assign that patient to the cluster corresponding to the highest membership probability. Among the 822 patients, 445 patients belong to the first cluster, 161 patients belong to the second cluster and the 216 patients belong to the third cluster. In \cref{dmt_plot}, we plot the MS DMT sequences of the patients corresponding to each cluster, up to the first 10 DMTs that they have received in their disease course. In case a patient has received fewer than 10 DMTs, the rest of the slots are kept blank. In the first cluster, most patients have been on interferon-beta, while a small number of patients have been on rituximab-ocrelizumab (B-cell Depletion), fingolimod and dimethyl fumarate. In the second cluster, patient-specific lengths of DMT sequences are smaller than that of the first and the third cluster. In this cluster, patients are mostly treated with glatiramer acetate, interferon-beta, teriflunomide, cyclophosphamide and dimethyl fumarate. In the third cluster, the majority of the patients have been on natalizumab. In summary, patients in the first and the third cluster received predominantly interferon-beta and natalizumab, respectively, whereas patients in cluster 2 received multiple DMTs without any predominant DMT. We also estimate the cluster-specific coefficient values corresponding to the clinical and demographic covariates considered which is provided in \cref{coeffs}.

\begin{table}[ht]
\centering
\resizebox{\columnwidth}{!}{
\begin{tabular}{|lllllll|}
\hline
Clusters & Intercept & Age at Diagnosis & Disease duration & \begin{tabular}[c]{@{}l@{}}Gender \\ (Female = 1)\end{tabular} & Race White & Race Black \\ \hline
Cluster 2 & $-2.48$ $(0.04)$ & $0.04$ $(0.02)$ & $-0.0$ $(0.01)$ & $0.29$ $(0.06)$ & $0.34$ $(0.05)$ & $-0.09$ $(0.05)$ \\
\multicolumn{1}{c}{Cluster 3} & $-0.17$ $(0.06)$ & $-0.02$ $(0.01)$ & $-0.04$ $(0.03)$ & $0.40$ $(0.11)$ & $0.75$ $(0.09)$ & $1.11$ $(0.08)$ \\ \hline
\end{tabular}}
\caption{Regression coefficient for clinical and demographical variables for cluster 2 and 3 taking cluster 1 as the reference cluster. Coefficients are estimated within proposed mixture Markov model with covariates analysis. Standard errors (calculated based on 50 bootstrap samples), is provided in the parenthesis.}
\label{coeffs}
\end{table}

Among the 822 patients, a subset of 488 patients belong to the CLIMB cohort. We compare clinical characteristics across the 3 clusters for the patients belonging to the CLIMB cohort which is noted down in \cref{yearly_rates}. CLIMB patients have more detailed clinical information. Among CLIMB cohort patients, Cluster 1, 2 and 3 account for 56\%, 20\% and 24\% of the patients, respectively. When comparing demographics, Cluster 2 patients from CLIMB cohort (pfCc) have the oldest mean age at diagnosis. Cluster 1 pfCc have the lowest whereas Cluster 3 pfCc have the highest proportion of women. When comparing MS outcomes during patient follow-up, Cluster 3 pfCc have the highest annualized relapse rate or ARR (registry data) as well as the highest annualized counts of MS ICD codes, MS-related MRI CPT codes, and MS-relevant CUIs (for example, ``multiple sclerosis.'', ``physical therapy''). Cluster 1 and 2 pfCc share similar mean ARR as well as the mean annualized counts of MS ICD codes and MS-related MRI CPT codes. When assessing the yearly incident relapse rate over time, all three clusters decline over time which is consistent with our prior finding \citep{Liang2022}, but Cluster 3 pfCc show consistently higher yearly relapse rate than the other two clusters (\cref{yir}). Cluster 3 pfCc also have the highest mean annualized counts of total ICD codes and total CPT codes, suggesting the highest healthcare utilization and comorbidity burden. Cluster 2 pfCc contains larger proportion of DMT ``cyclers'' (who experienced high frequency of switches to different DMTs) than Cluster 1 and 3 pfCc. For example, patients in Cluster 2 switched off standard-efficacy DMTs (for example, interferon-beta, glatiramer acetate), while patients in Cluster 3 switched to natalizumab, a higher-efficacy DMT which was approved in 2004. Cluster membership according to DMT prescription sequences correlate with key clinical outcomes such as yearly relapse rate (\cref{yir}).

\begin{figure}[H]
\centering
\includegraphics[width=0.95\textwidth]{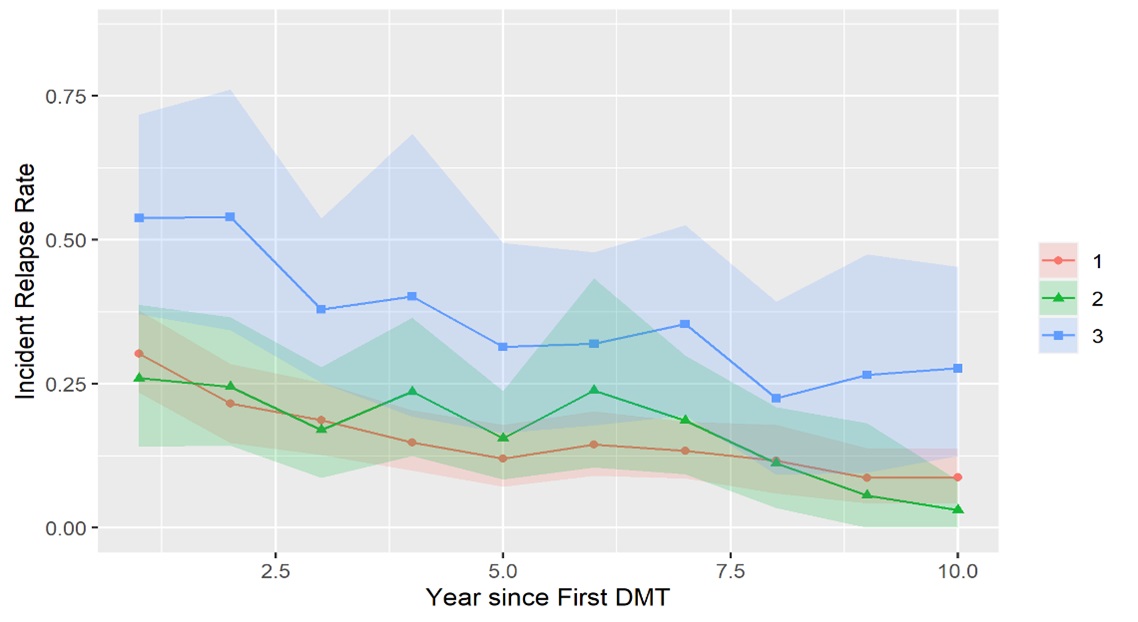}
\caption{Yearly incident relapse rate over time for patients from CLIMB cohort belonging to all three clusters.}
\label{yir}
\end{figure}

\begin{table}[]
\centering
\begin{tabular}{lccc}
\hline
Feature & Cluster 1 & Cluster 2 & Cluster 3 \\ \hline
Number of Patients & $271$ & $100$ & $117$ \\
Age at Diagnosis (year) & $36.5$ $(10.1)$ & $41.9$ $(10.7)$ & $34.1$ $(9.5)$ \\
Duration from Diagnosis to 1st DMT (year) & $0.9$ $(2.4)$ & $1.4$ $(2.9)$ & $1.5$ $(3.1)$ \\
Race-White & $90$\% & $94$\% & $93$\% \\
Race-Black & $4$\% & $3$\% & $7$\% \\
Race-Other & $6$\% & $3$\% & $0$\% \\
Gender-Male & $32$\% & $27$\% & $19$\% \\
Gender-Female & $68$\% & $73$\% & $81$\% \\
Annualized Relapse Rate & $0.16$ $(0.2)$ & $0.18$ $(0.3)$ & $0.33$ $(0.5)$ \\
ICD : Total (count/year) & $19.2$ $(20.0)$ & $26.2$ $(25.6)$ & $40.3$ $(42.2)$ \\
ICD : MS (count/year) & $11.8$ $(7.9)$ & $13.9$ $(10.4)$ & $24.2$ $(16.1)$ \\
CPT : Total (count/year) & $20.0$ $(16.0)$ & $25.5$ $(23.4)$ & $36.2$ $(41.1)$ \\
CPT : MS MRI (count/year) & $3.7$ $(1.9)$ & $3.4$ $(2.0)$ & $4.6$ $(2.3)$ \\
CPT : ED visit (count/year) & $0.7$ $(0.5)$ & $0.6$ $(0.6)$ & $0.8$ $(0.7)$ \\
CUI : Communicable disease (count/year) & $1.3$ $(2.5)$ & $1.7$ $(2.4)$ & $4.2$ $(4.1)$ \\
CUI : Double vision (count/year) & $0.6$ $(0.8)$ & $0.7$ $(1.0)$ & $0.6$ $(1.0)$ \\
CUI : Gadolinium (count/year) & $1.8$ $(1.0)$ & $1.9$ $(1.1)$ & $2.5$ $(1.5)$ \\
CUI : MRI (count/year) & $5.6$ $(2.3)$ & $5.9$ $(3.3)$ & $7.4$ $(3.8)$ \\
CUI : Methylprednisolone (count/year) & $1.3$ $(3.0)$ & $1.6$ $(3.0)$ & $2.6$ $(5.1)$ \\
CUI : Multiple sclerosis (count/year) & $7.6$ $(5.4)$ & $9.4$ $(7.5)$ & $15.0$ $(12.0)$ \\
CUI : Sensation loss (count/year) & $6.1$ $(3.5)$ & $6.7$ $(5.0)$ & $8.5$ $(6.3)$ \\
CUI : Nystagmus (count/year) & $1.0$ $(0.8)$ & $1.4$ $(1.1)$ & $1.3$ $(1.1)$ \\
CUI : Optic neuritis (count/year) & $0.4$ $(1.1)$ & $0.6$ $(1.6)$ & $0.6$ $(1.6)$ \\
CUI : Sense of pain (count/year) & $3.3$ $(4.1)$ & $4.7$ $(5.3)$ & $7.5$ $(9.3)$ \\
CUI : PET scan (count/year) & $3.0$ $(3.5)$ & $4.7$ $(5.5)$ & $7.9$ $(10.2)$ \\
CUI : Steroid (count/year) & $1.4$ $(1.5)$ & $1.6$ $(1.9)$ & $2.3$ $(2.9)$ \\
CUI : Recurrent disease (count/year) & $0.7$ $(0.8)$ & $0.9$ $(1.1)$ & $1.3$ $(1.9)$ \\
CUI : Physical therapy (count/year) & $3.1$ $(3.6)$ & $4.8$ $(5.4)$ & $8.1$ $(10.8)$ \\
CUI : Has difficulty doing (count/year) & $1.9$ $(1.5)$ & $2.4$ $(1.9)$ & $3.4$ $(4.6)$ \\
CUI : Migraine disorders (count/year) & $0.8$ $(0.8)$ & $0.9$ $(0.8)$ & $1.3$ $(1.6)$ \\
CUI : Flare (count/year) & $0.3$ $(0.6)$ & $0.4$ $(0.7)$ & $0.7$ $(1.6)$ \\
CUI : Tingling sensation (count/year) & $0.6$ $(0.9)$ & $0.8$ $(1.0)$ & $0.9$ $(1.3)$ \\ \hline
\end{tabular}
\caption{Characteristics of patient clusters based on DMT prescription sequence: \% or Mean (SD)}
\label{yearly_rates}
\end{table}

\section{Conclusion}
\label{sec_discussion}

In the context of estimating the parameters in the proposed model, we propose a novel Blackbox optimization technique MSiCOR based on Pattern Search (PS) for optimizing any function over the collection of unit-simplexes, where the simplexes can be of different dimensions. We also show that for convex functions, under some regularity conditions, MSiCOR converges to the global optimum point. Based on comparative study using several benchmark functions, MSiCOR is shown to outperform Genetic Algorithm (GA), Sequential Quadratic Programming (SQP) and Interior Point (IP) in terms of performance, in general. To maximize the likelihood of the proposed MMM with covariates, along with MSiCOR, we further use modified Recursive Modified Pattern Search (RMPS) for unconstrained optimization (provided in Section D of the Supplementary Material). Instead of updating the simplex constrained parameters and unconstrained parameters alternatively, we update all those parameters simultaneously using a joint algorithm and the iterative update steps of MSiCOR and updated RMPS for unconstrained optimization. 

Further, we proposed a novel MMM with covariates technique to cluster subjects based on their state sequence data along with clinically relevant patient-specific covariates. Using EHR data on treatment prescription sequence, the proposed method can cluster patients into clinically meaningful subgroups. The proposed model is useful for identifying the differential treatment sequences and trajectories in the patient population. Further, it informs how patient-specific covariates influence the treatment sequence and trajectory. Once the clusters are identified, membership probabilities for different clusters could be computed for future patients based on the estimated parameters. 

The simulation study shows that the estimated coefficients of the covariates for different clusters in the proposed MMM with covariates are close to the true values. Further for general MMM (without covariates), using MSiCOR, we obtain better results when compared to the EM algorithm. 

As a critical test case, we deploy the proposed method to cluster and analyze MS patients with available DMT and clinical and demographic covariates as a part of a well-characterized research cohort study. We identify 3 clusters of MS patients based on membership on the basis of DMT prescription sequence, where the first and the third cluster are enriched for interferon-beta and natalizumab, respectively, while the patients in the second cluster received multiple DMTs without a single predominant DMTs. Patients in different DMT sequence clusters exhibited different demographic and clinical characteristics. Notably, DMT sequence cluster informed differential clinical outcomes.
In the future, the proposed algorithm could be applied to other chronic diseases where medication sequence data and patient-specific covariate values are available. The proposed global optimization technique (combining MSiCOR and modified RMPS for unconstrained optimization) can also be applied for mixture Hidden Markov model analysis.
\vspace{15pt}

\noindent {\large\bf Supplementary Material}\\
\textbf{Additional results: } Supplementary material is made available \href{https://github.com/priyamdas2/MSiCOR/blob/main/MSiCOR_Supplementary_material.pdf}{\underline{here}}. \\
\textbf{Codes and dataset: } MATLAB codes and de-identified dataset are made available \href{https://github.com/priyamdas2/MSiCOR}{\underline{here}}.\\
\noindent {\large\bf Funding}
\begin{description}
\item Dr. Xia was supported by NINDS R01NS098023, NINDS R01NS124882.
\end{description}
\noindent {\large\bf Disclosure statement}
\begin{description}
\item The authors report there are no competing interests to declare. 
\end{description}
\bibliography{references}

\newpage

\end{document}